\documentclass[12pt,a4paper,reqno]{amsart}
\usepackage[all]{xy} 
\usepackage{amsmath,enumerate}
\usepackage{amsfonts}
\usepackage{amssymb}
\usepackage{mathrsfs}
\usepackage{amsthm}

\usepackage{tensor}
\usepackage{cancel}
\usepackage{tikz}

\usepackage{xcolor}
\usepackage{graphicx}

\usepackage{amsmath,calligra,mathrsfs}

\newcommand{\scHom}{\mathscr{H}\text{\kern -3pt {\calligra\large om}}}
\newcommand{\scExt}{\mathscr{E}\text{\kern -3pt {\calligra\large xt}}}

\DeclareMathAlphabet{\mathbbold}{U}{bbold}{m}{n}

\DeclareFontEncoding{OT2}{}{} 
\newcommand{\textcyr}[1]{%
 {\fontencoding{OT2}\fontfamily{wncyr}\fontseries{m}\fontshape{n}
 \selectfont #1}}
\newcommand{\Sha}{{\!\be\lbe\mbox{\textcyr{Sh}}}}

\theoremstyle{plain}
\newtheorem{theorem}{Theorem}[section]
\newtheorem{proposition}[theorem]{Proposition}
\newtheorem{lemma}[theorem]{Lemma}
\newtheorem{corollary}[theorem]{Corollary}
\newtheorem{assumption}[theorem]{Assumption}

\theoremstyle{definition}
\newtheorem{definition}[theorem]{Definition}

\theoremstyle{remark}
\newtheorem{remark}[theorem]{Remark}
\def\le{\kern 0.03em}

\def\a{\mathfrak{a}}

\def\Z{{\mathbb Z}}

\def\e{\kern 0.08em}
\def\be{\kern -.1em}
\def\lbe{\kern -.025em}

 \DeclareMathOperator{\rank}{rank}

\DeclareMathOperator{\ord}{ord}

\begin{document}
\title[Specialisations of the BCS main conjecture
to $\mathbb Z_p$-lines]{Specialisations of the Burungale--Castella--Skinner main conjecture
to $\mathbb Z_p$-lines}
\begin{abstract}
Let $p>3$ be a prime,  $E/\mathbb Q$ be an elliptic curve and $K$ an
imaginary quadratic field satisfying the hypotheses of
Burungale--Castella--Skinner, and let $L/K$ be the unique
$\mathbb{Z}_p^2$-extension. In this note, by combining their integral two-variable
main conjecture with the specialisation formula of the first-named
author, we obtain a characteristic-ideal identity over every
$\mathbb{Z}_p$-line in $L/K$, involving an explicit local factor.
With the characteristic ideal of a non-torsion module defined to be
zero, this identity also applies when the two-variable Perrin--Riou
element specialises to zero. We call such lines exceptional and prove that only finitely many occur. We also prove that the cyclotomic line is non-exceptional with
trivial local factor, and that the anticyclotomic line is exceptional.
Finally, we bound the number of exceptional lines by the cyclotomic
augmentation order and prove that if the $p$-primary Tate--Shafarevich group over
$K$ is finite and the cyclotomic $p$-adic height pairing is
non-degenerate, then this number is at most $\rank E(K)$. In particular,
when $E(K)$ has rank one, the anticyclotomic line is the unique
exceptional line.
\end{abstract}

\author{Ki-Seng Tan}
\address{Department of Mathematics\\
National Taiwan University\\
Taipei 10764, Taiwan}
\email{tan@math.ntu.edu.tw}
\author{Fabien Trihan}
\address{Sophia University,
Department of Information and Communication Sciences
7-1 Kioicho, Chiyoda-ku, Tokyo 102-8554, JAPAN}
\email{f-trihan-52m@sophia.ac.jp}

\author{Kwok-Wing Tsoi}
\address{Department of Mathematics\\
National Taiwan University\\
Taipei 10764, Taiwan}
\email{kwokwingtsoi@ntu.edu.tw}
\maketitle
\section{Introduction}

Recently, Burungale, Castella and Skinner \cite{BCS25} established, in
the residually irreducible setting and without imposing any auxiliary
ramification hypothesis on $E[p]$, new cases of Mazur's cyclotomic main
conjecture and Perrin--Riou's anticyclotomic Heegner-point main
conjecture, complementing the results of Castella,
Grossi and Skinner \cite{CGS23} at Eisenstein primes.  As an application, they also proved,
under an additional hypothesis recalled in \S\ref{sec:BCS}, an integral
two-variable main conjecture \cite[Thm.~1.4.1]{BCS25}. The aim of the
present note is, by using the specialisation formula of the first-named
author in \cite[Thm.~1]{Tan14}, to describe the specialisations of this
result to every one-dimensional quotient of the Galois group of the
relevant $\mathbb Z_p^2$-extension.

To formulate our result, let $E/\mathbb Q$ be an elliptic curve, 
$p>3$ be a prime at which $E$ has good ordinary reduction, and 
$K$ be an imaginary quadratic field. We assume throughout that the
triple $(E,p,K)$ satisfies the hypotheses of
Burungale--Castella--Skinner (see \S\ref{sec:BCS}). Let $L/K$
denote the unique $\mathbb Z_p^2$-extension of $K$, and set
\[
   \Gamma=\operatorname{Gal}(L/K)
   \qquad\text{and}\qquad
   \Lambda=\mathbb Z_p[[\Gamma]].
\]
For each intermediate extension $M/K$ of $L/K$, we write $X_M$ for the Pontryagin dual of the ordinary $p$-primary Selmer group of $E$ over $M$. We next write
\[
   \mathcal L^{\mathrm{PR}}_{E/L}\in\Lambda
\]
for the two-variable $p$-adic Rankin $L$-series of Perrin--Riou
\cite{PR88}, with the normalisation to be recalled in \S\ref{sec:BCS}. The two-variable main conjecture proved by Burungale, Castella and Skinner then gives an equality
\begin{equation}\label{eq:intro-BCS}
   \bigl(\mathcal L^{\mathrm{PR}}_{E/L}\bigr)
   =
   \operatorname{char}_{\Lambda}(X_L).
\end{equation}

Recall that a \emph{$\mathbb Z_p$-line} in $L/K$ is an intermediate
extension $L'/K\subset L/K$ such that
\[
   \Gamma':=\operatorname{Gal}(L'/K)\simeq\mathbb Z_p,
   \qquad
   \Lambda':=\mathbb Z_p[[\Gamma']].
\]
It is then natural to ask how the two-variable identity
\eqref{eq:intro-BCS} behaves upon specialisation to each such line. Indeed, this question is related to a general philosophy that
higher-dimensional main conjectures should encode compatible families
of one-variable main conjectures. A result in this direction,
albeit in a rather different setting, is proved by Burns in \cite{Bur15}. In positive characteristic, the first-named author constructs in
\cite{Tan26} $p$-adic $L$-functions over $\mathbb Z_p^d$-extensions
and proves specialisation formulae for them. These ideas are further developed
in \cite{TTT26} by the present authors, where the aforementioned specialisation formulae form part of the
essential input to the proof of an Iwasawa main conjecture for ordinary semistable
elliptic curves over global function fields, subject to a $\mu$-invariant
hypothesis. However, in the
present number-field setting, the passage from \eqref{eq:intro-BCS} to the
corresponding identities over $\mathbb Z_p$-lines is not formal, since
characteristic ideals do not in general commute with passage from
$\Lambda$ to $\Lambda'$.

To address this difficulty, we use a specialisation formula for
characteristic ideals of dual Selmer groups proved by the first-named
author in \cite[Thm.~1]{Tan14}. This formula describes the discrepancy explicitly
in terms of a global factor $\varrho_{L/L'}$ and a product of local factors. We prove
in Theorem~\ref{thm:Tan-specialisation} that, for every
$\mathbb Z_p$-line $L'/K\subset L/K$, one has
\[
   \varrho_{L/L'}=(1).
\]
Thus, in the present setting, the discrepancy under specialisation is
entirely accounted for by the local factors. Their product is a principal ideal
\[
   \vartheta_{L/L'}
   =
   \bigl(\theta_{L/L'}\bigr)
   \subset\Lambda'.
\]
The explicit choice of the generator $\theta_{L/L'}$ will be recalled in \S\ref{sec:linewise-specialisation}. The vanishing criterion in \cite[Lem.~1.8]{Tan14} implies that
$\theta_{L/L'}\neq 0$, and so one can define
\begin{equation}\label{eq:intro-linewise}
   \mathcal L_{E/L'}
   :=
   \theta_{L/L'}^{-1}\cdot 
   \operatorname{pr}_{L'}
      \bigl(\mathcal L^{\mathrm{PR}}_{E/L}\bigr)
   \in\operatorname{Frac}(\Lambda')
\end{equation}
where $\operatorname{pr}_{L'}\colon\Lambda\longrightarrow\Lambda'$
is the homomorphism induced by $\Gamma\twoheadrightarrow\Gamma'$. We refer to $\mathcal L_{E/L'}$ as the \emph{linewise element} associated to $L'$. 

The specialisations that vanish play a different arithmetic role. We shall therefore say that a $\mathbb Z_p$-line $L'/K$ is \emph{exceptional} (for $(E,p,K)$) if $\operatorname{pr}_{L'}
      \bigl(\mathcal L^{\mathrm{PR}}_{E/L}\bigr)=0$ and \emph{non-exceptional} otherwise. Following \cite{Tan14}, we take the characteristic ideal of a non-torsion module to be the zero ideal. With these conventions, the main result of this note is as follows.

Let $e(E,p,K)$ denote the number of exceptional lines, and let
$\operatorname{Reg}_p^{\mathrm{cyc}}(E/K)$ denote the regulator
of Schneider's cyclotomic \(p\)-adic height pairing (see \cite{Sch82}) on
$E(K)\otimes_{\mathbb Z}\mathbb Q_p$.

\begin{theorem}\label{thm:intro-main}
Fix a triple \((E,p,K)\) as above which satisfies Assumption~\ref{ass:BCS}. Let $L'/K$ be any $\mathbb Z_p$-line
in $L/K$.

\begin{enumerate}
\item[(i)]
The element $\mathcal L_{E/L'}$ belongs to $\Lambda'$ and satisfies
\[
   \bigl(\mathcal L_{E/L'}\bigr)
   =
   \operatorname{char}_{\Lambda'}(X_{L'}).
\]

\item[(ii)]
The following  are equivalent:
\begin{enumerate}
\item[(1)] $L'$ is exceptional;
\item[(2)] $\mathcal L_{E/L'}=0$;
\item[(3)] $X_{L'}$ is not torsion over $\Lambda'$.
\end{enumerate}

\item[(iii)]
The number \(e(E,p,K)\) is finite. i.e. there are only finitely many exceptional $\mathbb Z_p$-lines in $L/K$ for $(E,p,K)$.
\item[(iv)] Let $L^+$ and $L^-$ be the cyclotomic and anticyclotomic $\mathbb{Z}_p$-extensions of $K$ respectively. Then the following assertions hold.
\begin{enumerate}[(a)]
\item $\theta_{L/L^+}=1$ and $L^+/K$ is non-exceptional;
\item $L^-/K$ is exceptional.
\end{enumerate}
\item[(v)] If the $p$-primary part of the Tate--Shafarevich group $\Sha(E/K)[p^\infty]$ of $E$ over $K$  is finite and
$\operatorname{Reg}_p^{\mathrm{cyc}}(E/K)\neq0$, then
\[
   e(E,p,K)\leq\rank_{\mathbb Z}E(K).
\]
If, in addition, $\rank_{\mathbb Z}E(K)=1$, then $L^-/K$ is the
unique exceptional line.
\end{enumerate}
\end{theorem}

In particular, the discrepancy between the two characteristic ideals under
specialisation is measured precisely by the
local factor occurring in \cite[Thm.~1]{Tan14}. Our main theorem also gives a Selmer-theoretic interpretation
of exceptionality: a $\Z_p$-line is exceptional precisely when the corresponding ordinary Selmer module is non-torsion. Finally, we will show, in Theorem \ref{thm:linewise-specialisation}, that the linewise element retains the analytic information carried by the
two-variable Perrin--Riou element. More precisely, for finite-order characters in the range of the interpolation formula of
\cite[Thm.~2.2.1]{CGS23}, we derive an explicit interpolation formula for the linewise element in terms of central values of the corresponding complex Rankin $L$-functions.

In brief, the contents of this note are as follows. In \S\ref{sec:BCS}, we recall
the precise hypotheses of Burungale--Castella--Skinner, fix the
normalisations of the Selmer groups and $p$-adic $L$-functions, and
state the two-variable main conjecture in the form that we use. In \S\ref{sec:linewise-specialisation},
we recall the specialisation formula of \cite{Tan14}, describe
its global and local factors in the present setting, and prove  Theorem~\ref{thm:intro-main}(i) and (ii). We also
give there the interpolation formula for the linewise element. Finally, in \S4, we prove Theorem~\ref{thm:intro-main}(iii), study the cyclotomic and
anticyclotomic lines, and establish the conditional uniqueness
criterion in Theorem~\ref{thm:intro-main}(v).
\section{A review of the Burungale--Castella--Skinner theorem}
\label{sec:BCS}

We now make precise the hypotheses and normalisations used in the
article. Let $E/\mathbb Q$ be an elliptic curve of conductor $N$, let
$p>3$ be a prime, and let $K$ be an imaginary quadratic field of discriminant
$D_K$. We fix embeddings
\[
   \iota_\infty\colon\overline{\mathbb Q}\hookrightarrow\mathbb C,
   \qquad
   \iota_p\colon\overline{\mathbb Q}\hookrightarrow
                     \overline{\mathbb Q}_p,
\]
and use them to regard algebraic normalisations of complex $L$-values
as elements of $\overline{\mathbb Q}_p$. For any number field $F$, we
write $G_F=\operatorname{Gal}(\overline{\mathbb Q}/F)$. The following hypotheses, taken from
\cite[\S\S1.2 and~1.4]{BCS25}, will be in force throughout.

\begin{assumption}[BCS hypotheses]\label{ass:BCS}
The triple $(E,p,K)$ satisfies the following conditions:
\begin{itemize}
\item $E$ has good ordinary reduction at $p$;

\item $(\mathrm{disc})$ $D_K$ is odd and $D_K\neq -3$;

\item $(\mathrm{Heeg})$ every prime dividing $N$ splits in $K$;

\item $(\mathrm{spl})$ $p=v\bar v$
in $K$, where $v$ is the prime determined by $\iota_p$;

\item $(\mathrm{sur})$ the residual representation
\[
   \overline{\rho}_{E,p}\colon
   G_{\mathbb Q}\longrightarrow
   \operatorname{Aut}_{\mathbb F_p}(E[p])
\]
is surjective. Consequently, condition
\[
   (\mathrm{irr}_{\mathbb Q})\qquad
   E[p]\text{ is an irreducible }G_{\mathbb Q}\text{-module}
\]
also holds;

\item the set of `vexing primes'
\[
   V=
   \left\{
      \ell\equiv -1\pmod p\ \middle|\
      \left.\overline{\rho}_{E,p}\right|_{G_{\mathbb Q_\ell}}
          \text{ is irreducible and }
      \left.\overline{\rho}_{E,p}\right|_{I_\ell}
          \text{ is reducible}
   \right\}
\]
is empty, where $I_\ell\subset G_{\mathbb Q_\ell}$ denotes the inertia
subgroup.
\end{itemize}
\end{assumption}
\begin{remark}
In the terminology of \cite{BCS25}, Assumption~\ref{ass:BCS}
includes $(\mathrm{disc})$, $(\mathrm{Heeg})$, $(\mathrm{spl})$,
$(\mathrm{sur})$, and $V=\varnothing$. Since $(\mathrm{sur})$
implies $(\mathrm{irr}_{\mathbb Q})$, all the hypotheses of the
integral assertion \cite[Thm.~1.4.1(b)]{BCS25} are satisfied.
\end{remark}

For $M=L$ or $L'$, we put
\[
   X_M
   :=
   \operatorname{Hom}_{\mathbb Z_p}
   \bigl(\operatorname{Sel}_{p^\infty}(E/M),
         \mathbb Q_p/\mathbb Z_p\bigr)
\]
to be the Pontryagin dual of the $p$-primary Selmer group. This
definition agrees with that of \cite[p.~1025]{Tan14}. The extension
denoted $K_\infty/K$ in \cite[\S1.4]{BCS25} is the unique
$\mathbb Z_p^2$-extension of $K$. Hence $K_\infty=L$ and, by
definition, $X_L=X^{\mathrm{ord}}(E/K_\infty)$.
It is known that the modules $X_L$ and $X_{L'}$ are finitely generated over $\Lambda$
and $\Lambda'$, respectively (see
\cite[Prop.~1.1 and Cor.~2.14]{Tan14}).

Let $\mathcal L^{\mathrm{PR}}_{E/L}\in\Lambda$
denote the two-variable $p$-adic Rankin $L$-series constructed by
Perrin--Riou \cite{PR88}, with the normalisation of
\cite[Def.~2.2.2]{CGS23}. The integrality of this normalisation is
recalled in \cite[\S 1.4]{BCS25}. The result of
Burungale--Castella--Skinner that we shall use can now be stated as
follows.

\begin{theorem}[Burungale--Castella--Skinner]
\label{thm:BCS}
Under Assumption~\ref{ass:BCS}, the module $X_L$ is torsion over
$\Lambda$, and
\begin{equation}\label{eq:BCS-main}
   \operatorname{char}_{\Lambda}(X_L)
   =
   \bigl(\mathcal L^{\mathrm{PR}}_{E/L}\bigr).
\end{equation}
\end{theorem}

\begin{proof}
By $(\mathrm{sur})$, condition $(\mathrm{irr}_{\mathbb Q})$ holds.
The torsion assertion and the equality after tensoring with
$\mathbb Q_p$ therefore follow from \cite[Thm.~1.4.1(a)]{BCS25}.
The integral equality \eqref{eq:BCS-main} follows from
\cite[Thm.~1.4.1(b)]{BCS25}.
\end{proof}

We shall also use the interpolation property of
$\mathcal L^{\mathrm{PR}}_{E/L}$. Put
\[
   a_p=p+1-\#E(\mathbb F_p),
\]
and let $\alpha_p\in\mathbb Z_p^\times$ be the ordinary unit root of $X^2-a_pX+p.$

If $\omega_E$ is a N\'eron differential on $E$, set
\[
   \Omega_{E/K}
   :=
   \frac{1}{\sqrt{|D_K|}}
   \int_{E(\mathbb C)}\omega_E\wedge \mathrm{i}\overline{\omega_E}.
\]

Let $\eta\colon\Gamma\longrightarrow\overline{\mathbb Q}_p^{\,\times}$ be a finite-order character. For brevity, we call $\eta$
\emph{CGS-admissible} if either $\eta=1$, or its conductor has the form
\[
   v^m\bar v^{\,n}
   \qquad (m,n\geq 0,\;m+n>0).
\]
Thus a nontrivial everywhere-unramified finite-order character is not
CGS-admissible in this terminology.
Let
\[
  \operatorname{Art}^{\mathrm{ar}}_{L/K}\colon
  \mathbb A_K^\times/K^\times\longrightarrow\Gamma
\]
be the quotient of the global Artin map normalized so that uniformizers
map to arithmetic Frobenius.  For such a character, let
\[
  \bar\eta
  :=
  \iota_\infty\circ\iota_p^{-1}\circ\eta^{-1}
  \circ\operatorname{Art}^{\mathrm{ar}}_{L/K},
\]
a finite-order complex Hecke character of $K$, and set
$W(\eta):=W(\bar\eta)$ with $W$ as in \cite[\S2.2]{CGS23}. Define
\[
   I^{\mathrm{PR}}(\eta)
   :=
   \frac{W(\eta)\cdot p^{(m+n)/2}}
        {\alpha_p^{m+n}\cdot \Omega_{E/K}}
   \qquad\text{if $m+n>0$,}
\]
and
\[
   I^{\mathrm{PR}}(1)
   :=
   (1-\alpha_p^{-1})^4\cdot \Omega_{E/K}^{-1}.
\]
With these conventions, the interpolation formula of Perrin--Riou, in
the normalisation of Castella, Grossi and Skinner \cite{CGS23}, takes the following
form.

\begin{proposition}[Perrin--Riou--CGS interpolation]
\label{prop:PR-interpolation}
Under Assumption~\ref{ass:BCS}, every CGS-admissible finite-order
character $\eta$ satisfies
\[
   \eta\bigl(\mathcal L^{\mathrm{PR}}_{E/L}\bigr)
   =
   I^{\mathrm{PR}}(\eta)\cdot 
   L(E/K,\bar\eta,1).
\]
\end{proposition}

\begin{proof}
This follows from
\cite[Thm.~2.2.1, Def.~2.2.2 and Rem.~2.2.3]{CGS23}. The occurrence of
$\bar\eta$ in the complex $L$-value reflects our convention for
passing between finite-order $p$-adic characters and the corresponding
Hecke characters.
\end{proof}

\section{Specialisation to a $\mathbb Z_p$-line}
\label{sec:linewise-specialisation}

Let $L'/K\subset L/K$ be a $\mathbb Z_p$-line and put $\Psi=\operatorname{Gal}(L/L')$. For each finite place $w$ of $K$, let $\Gamma_w\subseteq\Gamma$
denote the decomposition group at a place of $L$ above $w$, and put
\[
   \Psi_w=\Psi\cap\Gamma_w.
\]
Since $\Gamma$ is abelian, these groups are independent of the choice
of the place above $w$.

The specialisation formula of the first-named author
\cite[Thm.~1]{Tan14} involves an explicit product of local ideals. We
first describe this factor in the present setting, following
\cite[Defs.~1.4, 1.5 and~1.7]{Tan14}. 

\begin{definition}[Local specialisation factor]
\label{def:local-factor}
The local factors and their product are defined as follows.

\begin{enumerate}
\item[(i)] \emph{Places away from $p$.}
For each finite place $w\nmid p$, let $\mathbb F_w$ be its residue
field and let $\Phi_w$ denote the group of connected components of
the special fibre of the N\'eron model of $E/K_w$. Put
\[
   c_w(E/K_w)=\#\Phi_w(\mathbb F_w),
   \qquad
   c_w^{(p)}=p^{v_p(c_w(E/K_w))}.
\]
Thus, $c_w^{(p)}$ is the $p$-part of the Tamagawa number
$c_w(E/K_w)$, and it generates the ideal $\pi_w$ of
\cite[Def.~1.4]{Tan14}.

\item[(ii)] \emph{Places above $p$.}
For each $w\mid p$, let $\alpha_w$ be the ordinary unit Frobenius
eigenvalue at $w$. Since $E$ is defined over $\mathbb Q$ and $p$
splits in $K$, one has
\[
   \alpha_v=\alpha_{\bar v}=\alpha_p\in\mathbb Z_p^\times.
\]
If $L'_w/K_w$ is unramified, let
\[
   [w]_{L'}\in\operatorname{Gal}(L'/K)=\Gamma'
\]
denote its Frobenius element. We then define
\[
   e_w(L')
   =
   \begin{cases}
   (1-\alpha_w^{-1}[w]_{L'})
   (1-\alpha_w^{-1}[w]_{L'}^{-1}),
      &\text{if $L'_w/K_w$ is unramified},\\[4pt]
   1, &\text{if $L'_w/K_w$ is ramified}.
   \end{cases}
\]

\item[(iii)] \emph{The total factor.}
We set
\[
   \theta_{L/L'}
   =
   \prod_{\substack{w\nmid p\\ \Psi_w\neq 0}}
      c_w^{(p)}
   \cdot
   \prod_{w\mid p}e_w(L')
   \in\Lambda'
\]
and define
\[
   \vartheta_{L/L'}
   =
   (\theta_{L/L'})
   \subseteq\Lambda'.
\]
Only finitely many factors in the first product are different from
$1$.
\end{enumerate}
\end{definition}

\begin{remark}
Notice that the
split-multiplicative case in \cite[Def.~1.7(c)]{Tan14} does not arise here,
since every place ramified in $L/K$ lies above $p$, where $E$ has
good ordinary reduction.
\end{remark}

The extension $L/K$ is unramified outside the primes above $p$, and $E$
has good ordinary reduction at every ramified place. Hence
\cite[Thm.~1]{Tan14} applies with $A=E$, $d=2$, and $e=1$.
The following result is the form of the specialisation theorem that we
shall use.

\begin{theorem}[Specialisation of characteristic ideals]
\label{thm:Tan-specialisation}
Under Assumption~\ref{ass:BCS}, the global factor in
\cite[Thm.~1]{Tan14} satisfies
\[
   \varrho_{L/L'}=(1).
\]
Consequently,
\[
   \operatorname{pr}_{L'}
      \bigl(\operatorname{char}_{\Lambda}(X_L)\bigr)
   =
   \vartheta_{L/L'}\,
\cdot    \operatorname{char}_{\Lambda'}(X_{L'}).
\]
Moreover, $\vartheta_{L/L'}\neq(0)$.
\end{theorem}

\begin{proof}
The specialisation formula of the first-named author
\cite[Thm.~1]{Tan14} gives
\begin{equation}\label{eq:Tan-general}
   \operatorname{char}_{\Lambda'}(X_{L'})
   \,\cdot \vartheta_{L/L'}
   =
   \varrho_{L/L'}\,\cdot 
   \operatorname{pr}_{L'}
      \bigl(\operatorname{char}_{\Lambda}(X_L)\bigr),
\end{equation}
where $\varrho_{L/L'}$ is the global factor defined in
\cite[Def.~1.3]{Tan14}. In the present case, this is the ideal associated
to the $\Lambda'$-module $T_p\bigl(E[p^\infty](L')\bigr),$ the $p$-adic Tate module of $E[p^\infty](L')$. We claim that
this module is zero. Indeed, condition $(\mathrm{sur})$ implies that the
image $H:=\overline{\rho}_{E,p}(G_K)$ is a normal subgroup of
$\operatorname{GL}_2(\mathbb F_p)$ of index at most $2$. Since $p>3$,
one has
\[
   [\operatorname{GL}_2(\mathbb F_p),
     \operatorname{GL}_2(\mathbb F_p)]
   =
   \operatorname{SL}_2(\mathbb F_p),
\]
and hence $H$ contains $\operatorname{SL}_2(\mathbb F_p)$.
In particular, $E[p]$ is irreducible as a $G_K$-module and the image
of $G_K$ on $E[p]$ is non-abelian.

Since $L'/K$ is Galois, $E[p](L')$ is a $G_K$-stable subspace of
$E[p]$. If it were nonzero, irreducibility would give $E[p](L')=E[p].$ The action of $G_K$ on $E[p]$ would then factor through the abelian
group $\operatorname{Gal}(L'/K)$, contradicting the non-abelianity of
its image. It follows that
\[
   E[p](L')=0
   \qquad\text{and hence}\qquad
   T_p\bigl(E[p^\infty](L')\bigr)=0.
\]
Thus $\varrho_{L/L'}=(1)$, and \eqref{eq:Tan-general} gives the
required identity.

Finally, $\vartheta_{L/L'}\neq(0)$ follows from
\cite[Lem.~1.8]{Tan14}. Indeed, the local factor can vanish only in
one of the split-multiplicative cases described there, whereas every
place ramified in $L/K$ lies above $p$, where $E$ has good
ordinary reduction.
\end{proof}

We shall also need the following consequence of the explicit
description of $\theta_{L/L'}$.

\begin{lemma}[Finite-order non-vanishing]
\label{lem:Tan-factor-nonvanishing}
Under Assumption~\ref{ass:BCS}, for every finite-order character $\eta\colon\Gamma'\longrightarrow
   \overline{\mathbb Q}_p^{\,\times}$, one has $\eta(\theta_{L/L'})\neq 0.$
\end{lemma}

\begin{proof}
The factors $c_w^{(p)}$ are nonzero constants, so it is enough to
consider the factors $e_w(L')$ for $w\mid p$. If $L'_w/K_w$ is
ramified, then $e_w(L')=1$. If it is unramified, then
\[
   e_w(L')
   =
   (1-\alpha_w^{-1}[w]_{L'})
   (1-\alpha_w^{-1}[w]_{L'}^{-1}).
\]
Since $\eta$ has finite order, $\eta([w]_{L'})$ is a root of
unity. On the other hand, $\alpha_w$ is not a root of unity by
\cite[(3)]{Tan14}. It follows that
\[
   \eta([w]_{L'})\neq\alpha_w
   \qquad\text{and}\qquad
   \eta([w]_{L'})\neq\alpha_w^{-1},
\]
and hence $\eta(e_w(L'))\neq0$. The result now follows from
Definition~\ref{def:local-factor}.
\end{proof}

We now combine Theorems~\ref{thm:BCS} and
\ref{thm:Tan-specialisation}. The following result proves
Theorem~\ref{thm:intro-main}(i) and also records the interpolation
property of the linewise element.

\begin{theorem}[Linewise specialisation]
\label{thm:linewise-specialisation}
Under Assumption~\ref{ass:BCS}, the following assertions hold.

\begin{enumerate}
\item[(i)]
The linewise element
\[
   \mathcal L_{E/L'}
   =
   \theta_{L/L'}^{-1}\cdot 
   \operatorname{pr}_{L'}
      \bigl(\mathcal L^{\mathrm{PR}}_{E/L}\bigr),
\]
initially defined in $\operatorname{Frac}(\Lambda')$, belongs to
$\Lambda'$ and satisfies
\[
   \bigl(\mathcal L_{E/L'}\bigr)
   =
   \operatorname{char}_{\Lambda'}(X_{L'}).
\]

\item[(ii)]
Let $\eta\colon\Gamma'\longrightarrow
   \overline{\mathbb Q}_p^{\,\times}$ be a finite-order character, and let $\widetilde{\eta}\colon\Gamma\longrightarrow
   \overline{\mathbb Q}_p^{\,\times}$ denote its inflation along the quotient map
$\Gamma\twoheadrightarrow\Gamma'$. If
$\widetilde{\eta}$ is CGS-admissible, then
\[
   \eta(\mathcal L_{E/L'})
   =
   \eta(\theta_{L/L'})^{-1}\cdot 
   I^{\mathrm{PR}}(\widetilde{\eta})\cdot 
   L(E/K,\overline{\widetilde{\eta}},1).
\]
\end{enumerate}
\end{theorem}

\begin{remark}
The equality in Theorem~\ref{thm:linewise-specialisation}(i) is
understood to include the zero ideal. More precisely, following the
convention of \cite{Tan14}, one has
\[
   \operatorname{char}_{\Lambda'}(X_{L'})=(0)
\]
whenever $X_{L'}$ is not torsion over $\Lambda'$. In this case,
Theorem~\ref{thm:linewise-specialisation}(i) asserts that
$\mathcal L_{E/L'}=0$. Thus the theorem applies uniformly to both
exceptional and non-exceptional lines, without any prior torsion
assumption on $X_{L'}$.
\end{remark}
\begin{proof}
We first prove (i). By Theorem~\ref{thm:BCS}, one has
\[
   \bigl(\mathcal L^{\mathrm{PR}}_{E/L}\bigr)
   =
   \operatorname{char}_{\Lambda}(X_L).
\]
Applying $\operatorname{pr}_{L'}$ and then using
Theorem~\ref{thm:Tan-specialisation} gives
\[
   \left(
      \operatorname{pr}_{L'}
         \bigl(\mathcal L^{\mathrm{PR}}_{E/L}\bigr)
   \right)
   =
   \vartheta_{L/L'}\,\cdot 
   \operatorname{char}_{\Lambda'}(X_{L'}).
\]
Since
\[
   \vartheta_{L/L'}=(\theta_{L/L'})
   \qquad\text{and}\qquad
   \theta_{L/L'}\neq0,
\]
we may divide this equality in
$\operatorname{Frac}(\Lambda')$ to obtain
\[
   \left(
      \theta_{L/L'}^{-1}\cdot 
      \operatorname{pr}_{L'}
         \bigl(\mathcal L^{\mathrm{PR}}_{E/L}\bigr)
   \right)
   =
   \operatorname{char}_{\Lambda'}(X_{L'}).
\]
The right-hand side is an ideal of $\Lambda'$, possibly the zero
ideal. It follows that
\[
   \theta_{L/L'}^{-1}\cdot 
   \operatorname{pr}_{L'}
      \bigl(\mathcal L^{\mathrm{PR}}_{E/L}\bigr)
   \in\Lambda'.
\]
By the definition of $\mathcal L_{E/L'}$, this proves both its
integrality and the equality
\[
   \bigl(\mathcal L_{E/L'}\bigr)
   =
   \operatorname{char}_{\Lambda'}(X_{L'}),
\]
and hence proves (i).

We next prove (ii). Let $\eta$ be as in that assertion. By
Lemma~\ref{lem:Tan-factor-nonvanishing}, one has
\[
   \eta(\theta_{L/L'})\neq0.
\]
Evaluating the definition of the linewise element at $\eta$
therefore gives
\[
   \eta(\mathcal L_{E/L'})
   =
   \eta(\theta_{L/L'})^{-1}
  \cdot  \eta\left(
      \operatorname{pr}_{L'}
         \bigl(\mathcal L^{\mathrm{PR}}_{E/L}\bigr)
   \right).
\]
Since $\widetilde{\eta}$ is the inflation of $\eta$ along
$\Gamma\twoheadrightarrow\Gamma'$, the corresponding evaluation maps
satisfy
\[
   \eta\left(
      \operatorname{pr}_{L'}
         \bigl(\mathcal L^{\mathrm{PR}}_{E/L}\bigr)
   \right)
   =
   \widetilde{\eta}
      \bigl(\mathcal L^{\mathrm{PR}}_{E/L}\bigr).
\]
As $\widetilde{\eta}$ is CGS-admissible,
Proposition~\ref{prop:PR-interpolation} implies that
\[
   \widetilde{\eta}
      \bigl(\mathcal L^{\mathrm{PR}}_{E/L}\bigr)
   =
   I^{\mathrm{PR}}(\widetilde{\eta})
   \cdot L(E/K,\overline{\widetilde{\eta}},1).
\]
Combining the preceding equalities gives
\[
   \eta(\mathcal L_{E/L'})
   =
   \eta(\theta_{L/L'})^{-1}
  \cdot  I^{\mathrm{PR}}(\widetilde{\eta})
   \cdot L(E/K,\overline{\widetilde{\eta}},1),
\]
as required.
\end{proof}

Theorem~\ref{thm:intro-main}(ii) is an immediate consequence.
\begin{corollary}
\label{cor:exceptional-equivalences}
Under Assumption~\ref{ass:BCS}, the following conditions are equivalent:
\begin{enumerate}
\item[(i)] $L'$ is exceptional;
\item[(ii)] $\mathcal L_{E/L'}=0$;
\item[(iii)] $X_{L'}$ is not torsion over $\Lambda'$.
\end{enumerate}
\end{corollary}

\begin{proof}
Since $\theta_{L/L'}\neq0$, the definition of the linewise element
gives
\[
   \mathcal L_{E/L'}=0
   \quad\Longleftrightarrow\quad
   \operatorname{pr}_{L'}
      \bigl(\mathcal L^{\mathrm{PR}}_{E/L}\bigr)=0.
\]
The right-hand condition is precisely the assertion that $L'$ is
exceptional. On the other hand,
Theorem~\ref{thm:linewise-specialisation} gives
\[
   \bigl(\mathcal L_{E/L'}\bigr)
   =
   \operatorname{char}_{\Lambda'}(X_{L'}).
\]
By the convention of \cite{Tan14}, this ideal is zero if and only if
$X_{L'}$ is not torsion over $\Lambda'$. This proves all the stated
equivalences.
\end{proof}

\section{Exceptional lines and distinguished specialisations}
\label{sec:exceptional-lines}

We first prove Theorem~1.1(iii). Although this also follows from \cite[Thm.~8]{Tan14}, in the present rank-two setting it admits the following direct proof.

\begin{proposition}[Finiteness of exceptional lines]
\label{prop:finiteness-exceptional}
Under Assumption~\ref{ass:BCS}, there are only finitely many
exceptional $\mathbb Z_p$-lines in $L/K$.
\end{proposition}

\begin{proof}
By Theorem~\ref{thm:BCS},
\[
   \bigl(\mathcal L^{\mathrm{PR}}_{E/L}\bigr)
   =
   \operatorname{char}_{\Lambda}(X_L),
\]
where $X_L$ is torsion over $\Lambda$. In particular, $\mathcal L^{\mathrm{PR}}_{E/L}\neq0.$

Let $L'/K$ be an exceptional line and put $ \Psi=\operatorname{Gal}(L/L').$ Then
\[
   \mathcal L^{\mathrm{PR}}_{E/L}
   \in
   \ker\bigl(\operatorname{pr}_{L'}\colon\Lambda\to\Lambda'\bigr).
\]
Since $\Gamma/\Psi\simeq\Gamma'\simeq\mathbb Z_p$ is torsion-free, $\Psi$ is a saturated rank-one
$\mathbb Z_p$-submodule of $\Gamma$. We may therefore choose a
topological generator $\psi$ of $\Psi$ and an element
$\gamma\in\Gamma$ such that
\[
   \Gamma
   =
   \langle\psi\rangle_{\mathbb Z_p}
   \oplus
   \langle\gamma\rangle_{\mathbb Z_p}.
\]
With respect to this basis, one has
\[
   \Lambda
   \simeq
   \mathbb Z_p[[\psi-1,\gamma-1]],
   \qquad
   \Lambda'
   \simeq
   \mathbb Z_p[[\gamma-1]],
\]
and hence $\ker(\operatorname{pr}_{L'})=(\psi-1)$. This is a height-one prime ideal of $\Lambda$.

Distinct $\mathbb Z_p$-lines give distinct kernels. Indeed, $\Psi$
can be recovered from the kernel by the formula
\[
   \Psi
   =
   \left\{
      \delta\in\Gamma\ \middle|\
      \delta-1\in\ker(\operatorname{pr}_{L'})
   \right\}.
\]
It follows that the exceptional lines give distinct height-one prime
ideals of $\Lambda$, each of which contains
$\mathcal L^{\mathrm{PR}}_{E/L}$.

Finally, since $\Lambda\simeq\mathbb Z_p[[T_1,T_2]]$
is a noetherian unique factorisation domain, a nonzero element  is contained in only finitely many height-one prime ideals.
Since $\mathcal L^{\mathrm{PR}}_{E/L}\neq0$, only finitely many
kernels of the form $\ker(\operatorname{pr}_{L'})$ can contain it.
The result follows.
\end{proof}

\subsection{The cyclotomic line}

Let
$\mathbb Q_\infty/\mathbb Q$ be the cyclotomic
$\mathbb Z_p$-extension and put $L^+=K\mathbb Q_\infty.$ We also set
\[
   \Gamma^+=\operatorname{Gal}(L^+/K),
   \qquad
   \Lambda^+=\mathbb Z_p[[\Gamma^+]],
\]
and write $\operatorname{pr}^{+}\colon\Lambda\longrightarrow\Lambda^+$ for the corresponding projection. Since $p$ is odd, one has $K\cap\mathbb Q_\infty=\mathbb Q$ and restriction induces an isomorphism
\[
   \Gamma^+
   \simeq
   \operatorname{Gal}(\mathbb Q_\infty/\mathbb Q).
\]

Let $E^K$ be the quadratic twist of $E$ by the character associated
to $K/\mathbb Q$. Via the preceding isomorphism, we regard the
Mazur--Swinnerton-Dyer $p$-adic $L$-functions
\[
   \mathcal L_p^{\mathrm{MSD}}(E/\mathbb Q),
   \qquad
   \mathcal L_p^{\mathrm{MSD}}(E^K/\mathbb Q)
\]
as elements of $\Lambda^+$, with the normalisation of
\cite[Thm.~2.1.1]{CGS23}.

\begin{proposition}[Cyclotomic specialisation]
\label{prop:cyclotomic-specialisation}
Under Assumption~\ref{ass:BCS}, one has
\[
   \theta_{L/L^+}=1.
\]
Moreover, there exists a unit $u^+\in(\Lambda^+)^\times$ such that
\[
   \operatorname{pr}^{+}
      \bigl(\mathcal L^{\mathrm{PR}}_{E/L}\bigr)
   =
   u^+\,\cdot 
   \mathcal L_p^{\mathrm{MSD}}(E/\mathbb Q)
   \cdot \mathcal L_p^{\mathrm{MSD}}(E^K/\mathbb Q).
\]
In particular, $L^+/K$ is non-exceptional and
\[
   \mathcal L_{E/L^+}
   =
   \operatorname{pr}^{+}
      \bigl(\mathcal L^{\mathrm{PR}}_{E/L}\bigr).
\]
\end{proposition}

\begin{proof}
We first compute the local factor. The two primes $v$ and $\bar v$
of $K$ above $p$ are ramified in $L^+/K$. Indeed, the corresponding
local extensions are cyclotomic $\mathbb Z_p$-extensions of
\[
   K_v=K_{\bar v}=\mathbb Q_p
\]
and are therefore totally ramified. It follows from
Definition~\ref{def:local-factor}(ii) that
\[
   e_v(L^+)=e_{\bar v}(L^+)=1.
\]

Now let $w\nmid p$ be a finite place of $K$, and put
\[
   q_w=\#\mathbb F_w.
\]
Since $L/K$ is unramified at $w$, the decomposition group
$\Gamma_w\subseteq\Gamma$ is procyclic and is topologically generated
by the Frobenius element at $w$. Under the isomorphism
\[
   \Gamma^+\simeq
   \operatorname{Gal}(\mathbb Q_\infty/\mathbb Q),
\]
the image of this Frobenius is the pro-$p$ component of
$q_w\in\mathbb Z_p^\times$. Since the rational integer $q_w>1$ is
not a $p$-adic root of unity, this component is nontrivial and hence
has infinite order in the torsion-free group $\Gamma^+$. The map
\[
   \Gamma_w\longrightarrow\Gamma^+
\]
is therefore injective. Consequently,
\[
   \Psi_w
   =
   \Gamma_w\cap\operatorname{Gal}(L/L^+)
   =
   \ker(\Gamma_w\longrightarrow\Gamma^+)
   =
   0.
\]
Thus no place $w\nmid p$ contributes to the first product in
Definition~\ref{def:local-factor}(iii), and hence $\theta_{L/L^+}=1.$

Assumption~\ref{ass:BCS} supplies the hypotheses of
\cite[Prop.~2.2.4]{CGS23}: the curve $E$ has good ordinary reduction
at $p$; since $p>3$ and $p\nmid N$, one has $p\nmid 2N$;
$(\mathrm{Heeg})$ implies $(D_K,N)=1$; and $(\mathrm{spl})$ gives
$p=v\bar v$ in $K$. Consequently, that proposition gives the
comparison formula
\[
   \operatorname{pr}^{+}
      \bigl(\mathcal L^{\mathrm{PR}}_{E/L}\bigr)
   =
   u^+\,\cdot 
   \mathcal L_p^{\mathrm{MSD}}(E/\mathbb Q)
   \cdot \mathcal L_p^{\mathrm{MSD}}(E^K/\mathbb Q)
\]
for some $u^+\in(\Lambda^+)^\times$. Since $p$ splits in $K$, the curves
$E$ and $E^K$ are isomorphic over $\mathbb Q_p$, so $E^K$ has good
ordinary reduction at $p$. Moreover,
$E^K[p]$ is irreducible, being a quadratic twist of $E[p]$. Thus
\cite[Thm.~1.1.2(a)]{BCS25}, applied to $E$ and $E^K$, shows that both
\[
   \mathcal L_p^{\mathrm{MSD}}(E/\mathbb Q)
   \quad\text{and}\quad
   \mathcal L_p^{\mathrm{MSD}}(E^K/\mathbb Q)
\]
are nonzero. Since $\Lambda^+$ is an integral domain, their product
is nonzero. It follows that
\[
   \operatorname{pr}^{+}
      \bigl(\mathcal L^{\mathrm{PR}}_{E/L}\bigr)\neq0,
\]
and hence $L^+$ is non-exceptional. Finally,
$\theta_{L/L^+}=1$ gives
\[
   \mathcal L_{E/L^+}
   =
   \operatorname{pr}^{+}
      \bigl(\mathcal L^{\mathrm{PR}}_{E/L}\bigr).
\]
\end{proof}

\begin{corollary}[Cyclotomic characteristic ideal]
\label{cor:cyclotomic-characteristic}
Under Assumption~\ref{ass:BCS}, one has
\[
   \operatorname{char}_{\Lambda^+}(X_{L^+})
   =
   \left(
      \mathcal L_p^{\mathrm{MSD}}(E/\mathbb Q)
      \cdot \mathcal L_p^{\mathrm{MSD}}(E^K/\mathbb Q)
   \right).
\]
\end{corollary}

\begin{proof}
By Theorem~\ref{thm:linewise-specialisation}(i),
\[
   \operatorname{char}_{\Lambda^+}(X_{L^+})
   =
   \bigl(\mathcal L_{E/L^+}\bigr).
\]
Proposition~\ref{prop:cyclotomic-specialisation} identifies the latter
ideal with
\[
   \left(
      u^+\,
      \cdot \mathcal L_p^{\mathrm{MSD}}(E/\mathbb Q)
      \cdot \mathcal L_p^{\mathrm{MSD}}(E^K/\mathbb Q)
   \right).
\]
Since $u^+$ is a unit in $\Lambda^+$, the result follows.
\end{proof}

\subsection{The anticyclotomic line}

Let $L^-/K$ be the
anticyclotomic $\mathbb Z_p$-extension of $K$, and set
\[
   \Gamma^-=\operatorname{Gal}(L^-/K),
   \qquad
   \Lambda^-=\mathbb Z_p[[\Gamma^-]].
\]
We write $\operatorname{pr}^{-}\colon\Lambda\longrightarrow\Lambda^-$ for the corresponding projection.

\begin{proposition}[Anticyclotomic exceptionality]
\label{prop:anticyclotomic-exceptionality}
Under Assumption~\ref{ass:BCS}, the line $L^-/K$ is exceptional. More
precisely,
\[
   \operatorname{pr}^{-}
      \bigl(\mathcal L^{\mathrm{PR}}_{E/L}\bigr)=0,
\]
and $X_{L^-}$ is not torsion over $\Lambda^-$.
\end{proposition}

\begin{proof}
By \cite[Thm.~1.2.2(a)]{BCS25}, the module
\[
   X_{\mathrm{ord}}(E/L^-),
\]
which is $X_{L^-}$ in our notation, has rank one over $\Lambda^-$.
It is therefore not torsion. Corollary
\ref{cor:exceptional-equivalences} now implies that $L^-$ is
exceptional and hence that
\[
   \operatorname{pr}^{-}
      \bigl(\mathcal L^{\mathrm{PR}}_{E/L}\bigr)=0.
\]
\end{proof}

\subsection{Conditional uniqueness of the exceptional line}

Let
\[
   F=\mathcal L^{\mathrm{PR}}_{E/L},
   \qquad
   \a=\ker(\Lambda\longrightarrow\mathbb Z_p),
   \qquad
   J^+=\ker(\Lambda^+\longrightarrow\mathbb Z_p)
\]
be the two-variable Perrin--Riou element and the augmentation ideals
of $\Lambda$ and $\Lambda^+$, respectively. For
$0\neq h\in\Lambda^+$, set
\[
   \ord_{J^+}(h)=\max\{n\geq0:h\in(J^+)^n\}.
\]
Recall that \(e(E,p,K)\) denotes the number of exceptional
\(\mathbb Z_p\)-lines in \(L/K\), which is finite by
Proposition~\ref{prop:finiteness-exceptional}.
\begin{proposition}[Cyclotomic bound]
\label{prop:cyclotomic-bound-exceptional}
Under Assumption~\ref{ass:BCS}, one has
\[
   F\in\a^{e(E,p,K)}
   \qquad\text{and}\qquad
   \mathcal L_{E/L^+}\in(J^+)^{e(E,p,K)}.
\]
Consequently,
\[
   e(E,p,K)
   \leq
   \ord_{J^+}\bigl(\mathcal L_{E/L^+}\bigr).
\]
\end{proposition}

\begin{proof}
Let $L_1,\ldots,L_e$ be the exceptional lines, where
$e=e(E,p,K)$. For each $i$, choose a topological generator $\psi_i$
of $\operatorname{Gal}(L/L_i)$. The proof of
Proposition~\ref{prop:finiteness-exceptional} gives
\[
   \ker(\operatorname{pr}_{L_i})=(\psi_i-1),
\]
and these are pairwise distinct height-one prime ideals of the unique
factorisation domain $\Lambda$. Since $L_i$ is exceptional,
$\psi_i-1$ divides $F$. Hence
\[
   \prod_{i=1}^{e}(\psi_i-1)\mid F.
\]
Every factor $\psi_i-1$ belongs to $\a$, and therefore
$F\in\a^e$. The cyclotomic projection sends $\a$ into $J^+$.
Proposition~\ref{prop:cyclotomic-specialisation} gives
\[
   \operatorname{pr}^+(F)=\mathcal L_{E/L^+}\neq0.
\]
It follows that
$\mathcal L_{E/L^+}\in(J^+)^e$, which proves the asserted bound.
\end{proof}

\begin{lemma}
\label{lem:EK-p-torsion}
Under Assumption~\ref{ass:BCS}, one has $E(K)[p]=0$.
\end{lemma}

\begin{proof}
Since $G_K$ is normal in $G_{\mathbb Q}$, the subspace
$E[p]^{G_K}$ is stable under $G_{\mathbb Q}$. Assumption
$(\mathrm{sur})$ implies that $E[p]$ is an irreducible
$G_{\mathbb Q}$-module. Hence, if $E[p]^{G_K}\neq0$, then $G_K$
acts trivially on $E[p]$. The residual representation would then
factor through $\operatorname{Gal}(K/\mathbb Q)$, contradicting its
surjectivity onto $\operatorname{GL}_2(\mathbb F_p)$.
\end{proof}

\begin{theorem}
\label{thm:conditional-uniqueness}

Under Assumption~\ref{ass:BCS}, suppose that
$\Sha(E/K)[p^\infty]$ is finite and that
$\operatorname{Reg}_p^{\mathrm{cyc}}(E/K)\neq0$. Then
\[
   \ord_{J^+}\bigl(\mathcal L_{E/L^+}\bigr)
   =
   \rank_{\mathbb Z}E(K)
\]
and
\[
   e(E,p,K)\leq\rank_{\mathbb Z}E(K).
\]
Moreover, $X_{L^+}$ is semisimple at $J^+$, in the sense that
\[
   J^+(X_{L^+})_{J^+}=0.
\]
If $\rank_{\mathbb Z}E(K)=1$, then the anticyclotomic line $L^-/K$
is the unique exceptional line.
\end{theorem}

\begin{proof}
The curve $E$ has good ordinary reduction at every prime of $K$
above $p$, and Lemma~\ref{lem:EK-p-torsion} gives $E(K)[p]=0$.
Schneider's cyclotomic leading-term theorem
\cite{Schneider85}, in the form stated in
\cite[Thm.~3.2]{Ray23}, applies to a generator $f^+$ of
$\operatorname{char}_{\Lambda^+}(X_{L^+})$ and gives
\[
   \ord_{J^+}(f^+)
   =
   \rank_{\mathbb Z}E(K).
\]
Theorem~\ref{thm:linewise-specialisation}(i) gives
\[
   (f^+)=\bigl(\mathcal L_{E/L^+}\bigr).
\]
The first equality follows, and
Proposition~\ref{prop:cyclotomic-bound-exceptional} gives the stated
bound.

Set
\[
   R=(\Lambda^+)_{J^+},
   \qquad
   X^+=(X_{L^+})_{J^+},
   \qquad
   r=\rank_{\mathbb Z}E(K).
\]
The ring $R$ is a discrete valuation ring with maximal ideal $J^+R$
and residue field $\mathbb Q_p$. Mazur's cyclotomic control theorem
\cite{Mazur72}, in the exact restriction-map form recorded in
\cite[p.~59]{Ochiai01}, together with the Kummer exact sequence and
the finiteness of $\Sha(E/K)[p^\infty]$, gives
\[
   \dim_{\mathbb Q_p}\bigl(X^+/J^+X^+\bigr)=r.
\]
Since $X^+$ has finite length over $R$, there are positive integers
$a_1,\ldots,a_s$ such that
\[
   X^+\simeq\bigoplus_{i=1}^s R/(J^+R)^{a_i}.
\]
The preceding dimension formula gives $s=r$, while Schneider's
leading-term theorem gives
\[
   \sum_{i=1}^s a_i
   =\operatorname{length}_R(X^+)
   =r.
\]
Thus every $a_i$ is equal to one and $J^+X^+=0$.

If $\rank_{\mathbb Z}E(K)=1$, the bound shows that there is at
most one exceptional line. Proposition
\ref{prop:anticyclotomic-exceptionality} shows that $L^-/K$ is
exceptional, so it is the unique exceptional line.
\end{proof}

\begin{corollary}[Analytic rank one]
\label{cor:analytic-rank-one-uniqueness}
Under Assumption~\ref{ass:BCS}, suppose that
\[
   \ord_{s=1}L(E/K,s)=1
   \qquad\text{and}\qquad
   \operatorname{Reg}_p^{\mathrm{cyc}}(E/K)\neq0.
\]
Then the anticyclotomic line $L^-/K$ is the unique exceptional line.
\end{corollary}

\begin{proof}
The Gross--Zagier formula
\cite[Ch.~I, (6.5)]{GrossZagier86}, as recorded in
\cite[p.~236, (1.1)]{Gross91}, shows that the Heegner point over $K$
has infinite order. Kolyvagin's theorem
\cite[Thm.~1.3]{Gross91} then gives
$\rank_{\mathbb Z}E(K)=1$ and the finiteness of $\Sha(E/K)$.
Theorem~\ref{thm:conditional-uniqueness} applies.
\end{proof}

\section*{Acknowledgements} The authors gratefully acknowledge the financial support from the Mathematics Division of the National Center for Theoretical Sciences (NCTS), and the stimulating and hospitable working environment that it provides, in which part of this work was carried out. They are especially grateful to Professor Ming-Lun Hsieh, Director of the Division, for his generosity in making this support available and for his warm encouragement of their work.

\end{document}